\documentclass{amsart}
\usepackage{fullpage}

\usepackage{amsfonts}
\usepackage{amsmath, amscd, amssymb}

\newtheorem{theorem}{Theorem}

\newtheorem{lemma}[theorem]{Lemma}

\theoremstyle{remark}

\title{On the connectivity of finite subset spaces}
\author{J.~Mostovoy}
\author{R.~Sadykov}


\begin{document}
\begin{abstract} 
We prove that the space $\exp_k \vee S^{m+1}$ of non-empty subsets of cardinality at most $k$ in a bouquet of $m+1$-dimensional spheres  is $(m+k-2)$-connected. This, as shown by Tuffley, implies that  the  space $\exp_k X$ is $(m+k-2)$-connected for any $m$-connected cell complex $X$.
\end{abstract}
\maketitle

\section{Introduction and the statement of the result}

The $k$-th finite subset space $\exp_k X$ of a topological space $X$ is the space of nonempty subsets of $X$ of cardinality at most $k$. The topology is taken to be the quotient topology with respect to the map $X^{k}\to \exp_k X$ that sends $(x_1,\ldots, x_k)$ to the subset $\{x_1\}\cup \ldots\cup\{ x_k\}\subset X$.  The space $\exp_k X$ can be interpreted as the space of faces of dimension at most $k$ of the infinite-dimensional simplex whose space of vertices is $X$. It is easy to see that $\exp_k$ is a homotopy functor.

The finite subset space functor was introduced by Borsuk and Ulam \cite{BU} as a means to construct examples of topological spaces with interesting properties. This construction can produce rather non-trivial results even in the simplest cases \cite{B, Bt, M}.  Finite subset spaces turned out to be of importance in various problems of geometry and topology, see, for instance, \cite{BD} and \cite{G}. 

The most important property of the finite subset spaces is the fact that the colimit $$\exp X = \cup \exp_k X$$ has trivial homotopy groups. This, essentially, follows from the fact that $\exp X$ is a topological monoid with an idempotent operation (see page 172 of \cite{BD} for a brief proof). In contrast to $\exp X$, the spaces $\exp_k X$ are not necessarily contractible. In this note we prove a lower bound on the connectivity of $\exp_k X$.

\begin{theorem}\label{main}
 If $X$ is an $m$-connected cell complex,  then the finite subset space $\exp_k X$ is $(k+m-2)$-connected. 
\end{theorem}

This was proved for $m = 0$ and $1$ and conjectured for $m \geq 2$ by Tuffley \cite{T2}.

Theorem~\ref{main} has also been established by Kallel and Sjerve \cite{KS1} in the case $k=3$, by F\'elix and Tanr\'e \cite{FT} rationally for k=3 and 4, and by Taamallah \cite{T} for $k=4$.

\section{The proof}

It has been shown by Tuffley~\cite{T2} that it is sufficient to prove Theorem~\ref{main} for finite bouquets of $(m+1)$-dimensional spheres. He also established it for bouquets of circles. Namely, he showed \cite{T1} that for a connected graph $\Gamma$ the space $\exp_k \Gamma$ 
has vanishing reduced homology groups in degrees different from $k-1$ and $k$. Since $\exp_k X$ is  simply connected for all connected $X$ and $k>2$
(see \cite[ Theorem~1]{T2} or \cite[Corollary~2.2]{KS1}), this implies that  $\exp_k\Gamma$ is ($k-2$)-connected. 

The proof for the bouquets of spheres of any given dimension is by induction on the dimension of the spheres, the base being the case of the bouquets of circles proved by Tuffley. Our argument is very close to that of Tuffley and uses the following simple lemma 
(which is a stronger version of Lemma~1 of \cite{T2}):
\begin{lemma}\label{l:1} Let $Y$ be a union of open sets $U_1, \dots, U_r$
such that 
\begin{itemize}
\item $U_1\cap \cdots \cap U_r$ is nonempty, 
\item each intersection $U_{i_1...i_s}=
U_{i_1}\cap \cdots \cap U_{i_s}$ has vanishing reduced homology in dimensions less than $j$, and
\item each $U_i$ has vanishing reduced homology in dimensions less than $j+1$. 
\end{itemize}
Then $Y$ has vanishing reduced homology in dimensions less than $j+1$. 
\end{lemma}
\begin{proof}
Observe that $U_{1}\cup U_2$ has vanishing reduced homology in dimensions less than 
$j+1$. Indeed, this immediately follows from the Mayer-Vietoris sequence 
\[
    \ldots\longrightarrow\widetilde{H}_{i+1}(U_{1})\oplus \widetilde{H}_{i+1}(U_2) \longrightarrow \widetilde{H}_{i+1}(U_{1}\cup U_2) \longrightarrow 
    \widetilde{H}_i(U_{12}) \longrightarrow \ldots 
\]
Next we apply the same argument to the open cover $\{V_i\}$ of $Y$, where
 $V_1=U_1\cup U_2$, $V_2=U_3$, ..., $V_{r-1}=U_r$, 
and prove the lemma by induction. 
\end{proof}

Assume that Theorem~\ref{main} has been established for all $k$ and $m\le n$. Let us prove it for all $k$ and $m=n+1$. 
As mentioned before, it suffices to consider the case of finite subset spaces $\exp_k \vee S^{n+1}$.

Let $P_1, ..., P_{k+1}$ be disjoint subsets of a finite bouquet $\vee S^{n+1}$ of $n+1$-dimensional spheres such that each $P_i$ 
has exactly one point in each sphere, and the common point of the spheres does not belong to any of the sets $P_i$. Then the open sets
\[
   U_i = \exp_k(\vee S^{n+1}-P_i)
\] 
cover $\exp_k\vee S^{n+1}$. Since $\exp_k$ is a homotopy functor, it follows that each $U_i$ is contractible and the intersections
\[
  U_{i_1...i_s}=\exp_k(\vee S^{n+1}-P_{i_1}-\cdots - P_{i_s})\simeq
  \exp_k(\vee S^{n})
\]
are  $(n+k-3)$-connected by the induction assumption. By Lemma~\ref{l:1}, we deduce that 
$\exp_k\vee S^{n+1}$ is  $(n+k-2)$-connected. 

\section{Some remarks}

The space $\exp X$ can  be thought of as a space of particles with summable labels (such as those in \cite{K}), with the monoid of labels being the set $\{0,\infty\}$ endowed with the natural addition. There is a whole family of monoids interpolating between this monoid and the natural numbers: take the set $\{0,1,2,\ldots,n, \infty\}$ with the operation being the usual sum unless the result is greater than $n$, in which case it is taken to be $\infty$. Particle spaces with labels in these monoids are also readily seen to be contractible.

\medskip

 The functor $\exp_k$ is actually defined on sets. For a set $S$ the set $\exp_k S$ is a colimit of a diagram consisting of all cartesian products of at most $k$ copies of $S$ with arrows being (1) all possible products of diagonal maps with identity maps and (2) all permutations of the factors.  If $S$ is given a topology, these arrows are continuous maps and, therefore, $\exp_k S$ is a topological space. If $S$ is a simplicial set,  that is, a functor from $\Delta^{\rm op}$ to sets, $\exp_k S$ can be defined simply as the composition $\exp_k \circ S$. Note that since the geometric realization commutes both with the cartesian products and the colimits, for any simplicial set $S$ the geometric realization  $|\exp_k S|$ is homeomorphic to $\exp_k |S|$. This fact can be used to produce cell decompositions of the spaces $\exp_k X$, such as those in \cite{KS1}.  (See \cite{GJ} for basics on simplicial sets.)

\end{document}